\documentclass[11pt]{article}
\input{preamble.tex}

\bibliography{refs}

\title{Cartesian Factorization Systems and Grothendieck Fibrations}
\author{David Jaz Myers}
\date{April 2020}

\begin{document}

\maketitle

\begin{abstract}
Every Grothendieck fibration gives rise to a vertical/cartesian orthogonal
factorization system on its domain. We define a \emph{cartesian factorization
  system} to be an orthogonal factorization in which the left class satisfies
2-of-3 and is closed under pullback along the right class. We endeavor to show
that this definition abstracts crucial features of the vertical/cartesian factorization
system associated to a Grothendieck fibration, and give comparisons between
various 2-categories of factorization systems and Grothendieck fibrations to
demonstrate this relationship. We then give a construction which corresponds to
the fiberwise opposite of a Grothendieck fibration on the level of cartesian
factorization systems.

Apart from the final double categorical results, this paper is entirely review
of previously established material. It should be read as an expository note.
\end{abstract}

\begin{rmk}
  After posting an earlier draft on the arXiv, I was notified by a number of
  people that the results in this paper have appeared in various forms
  elsewhere. The current draft has been updated in that light. Apart from the
  Grothendieck double construction and very final theorem (to my knowledge), all
  the definitions and theorems have appeared in the literature; this 
  paper, then, can be read as an expository account. I would like to thank Jonas
  Frey, Fosco Loregian, Richard Garner, Edoardo Lanari, and Charles Barwick for
  bringing these references to my attention.
\end{rmk}

A Grothendieck fibration $p : \Ea \to \Ba$ is a way of packaging the data of an
indexed category $\Ea_B$ varying with $B \in \Ba$. There is an involution on
indexed categories giving by taking the pointwise dual: namely, $\Ea_B \mapsto
\Ea_B\op$. What does this involution look like on the corresponding Grothendieck
fibrations?

The Grothendieck construction of the fiberwise opposite of an indexed category
is called the category of \emph{generlized lenses} by Spivak
\cite{spivak2019generalized}. It has uses in a general theory of open dynamical
systems \cite{jaz2020dynamical}.

A hint of how this construction should proceed is given by the theory of polynomial functors \cite{gambino_kock_2013}. A polynomial functor $P :
\Set \to \Set$ is a functor of the form
$$P(X) = \sum_{b \in B} X^{E_b}$$
for some family of sets $E_b$ varying with $b \in B$. We may see this family as
a function $p : E \to B$, and the data of a polynomial functor is precisely the
data of such a function. Any natural transformation between polynomial functors
$P$ and $P'$ may be represented by an odd diagram of the following sort, called
a morphism of polynomials:
\[
\begin{tikzcd}
E \arrow[d, "p"'] & \bullet \arrow[d] \arrow[r] \arrow[l] \arrow[rd, phantom,
"\ulcorner" very near start] & E' \arrow[d, "p'"] \\
B                 & B \arrow[r] \arrow[l, equals]                            & B'                
\end{tikzcd}
\]

Thinking only of the functions $p : E \to B$, we would expect a morphism $p \to
p'$ would be a square as below right, which is equivalent to a diagram
as below right:
\[
\begin{tikzcd}
E \arrow[d, "p"'] \arrow[r] & \bullet \arrow[d] \arrow[r] \arrow[rd, phantom,
"\ulcorner" very near start] & E' \arrow[d, "p'"] &  & E \arrow[d, "p"'] \arrow[r] & E' \arrow[d, "p'"] \\
B \arrow[r, equals]                 & B \arrow[r]                            & B'                 &  & B \arrow[r]                 & B'                
\end{tikzcd}
\]
The shape of this diagram on the right is explained by the vertical/cartesian factorization system on
$\Set^{\down}$ associated to the codomain fibration $\term{cod} : \Set^{\down}
\to \Set$. Namely, the left class of this factorization system consists of all commuting
squares whose bottom face is an isomorphism, and the right class consists of all
pullback squares. The morphisms of polynomials are precisely the spans in
$\Set^{\down}$ whose left leg is vertical and whose right leg is cartesian.

Thinking fiberwise, in terms of indexed categories, a morphism of functions $p \to p'$ can be expressed as a pair $f : B \to B'$ and
a family $f^{\sharp}_b : E_b \to E'_{f(b)}$ for $b \in B$. On the other hand, a morphism of polynomials is given by a pair of a function $f : B \to B'$ and a family $f^{\sharp}_b :
E'_{f(b)} \to E_b$ for $b \in B$. We can see, therefore, that forming these
spans of squares whose left leg is vertical and right leg is cartesian (the
polynomial morphisms) corresponds to taking the fiberwise opposite. We will see
that this construction works generally.

Every Grothendieck fibration $p : \Ea \to \Ba$ gives rise to such an orthogonal
factorization system, where the vertical maps are those sent to isomorphisms by
$p$ and the cartesian maps are those satisfying the cartesian lifting property.
We will show that we can perform a construction analogous to the construction of
the category of polynomials as spans whose left leg is vertical and whose right
leg is cartesian from the codomain fibration $\term{cod} :
\Set^{\down} \to \Set$ for any Grothendieck fibration --- without any extraneous
assumptions on the categories $\Ea$ and $\Ba$.

In this paper, we will see an abstraction of the crucial features of the vertical/cartesian
factorization systems associated to Grothendieck fibrations in the notion of a \emph{cartesian factorization system}: an
orthogonal factorization system $(\epi, \mono)$ satsifying two additional
properties:
\begin{enumerate}
\item The left class $\epi$ satisfies 2-of-3.
\item Pullbacks of the left class $\epi$ along the right class $\mono$ exist and
  are in $\epi$.
\end{enumerate}

Such factorization systems were studied first (to the author's knowledge) in
\cite{cassidy_hébert_kelly_1985} as a special sort of ``reflective factorization system'' (see
Theorem 4.3), with the specific case being called ``semi-left-exact'' reflective
factorization systems, and the connection to Grothendieck fibrations was later made in
\cite{rosicky_tholen_2007} where they are called ``simple reflective factorization systems''.
They have been called ``cartesian factorization systems'' in the
$\infty$-category literature (see \cite{lanari2019cartesian} and \cite{barwick_glasman_nardin_2014}).

Every vertical/cartesian factorization system arising from a Grothendieck
fibration is a cartesian factorization system. While we do not show that every
cartesian factorization system arises this way, we will show that under mild
conditions (such as the existence of a terminal object) they do. In particular,
we will show that if a cartesian factorization system admits enough
$\epi$-injectives (Definition \ref{defn:enough.injectives}), then we can
construct a Grothendieck fibration with a right adjoint right inverse from it.
We prove in Theorem \ref{thm:cartesian.equivalence} that this construction gives
an equivalence of the a 2-category of cartesian factorization systems
with enough injectives and a 2-category of Grothendieck fibrations with right adjoint right inverses.

More precisely, we will show that these factorization systems give rise to
\emph{Street fibrations}, which are the equivalence invariant twin of
Grothendieck fibrations.\footnote{Or, we might say that Grothendieck fibrations
  are the evil twin of Street fibrations.} These differ only in that the lifting
is only up to isomorphism; however, the universal property of cartesian maps
assures that these isomorphisms are unique. See \cite{bourn_shift_1987}, Part II Section 3
Proposition 1 for more detail.

We will then describe the \emph{fiberwise dual} construction, which mimics the
definition of polynomial morphisms in a general cartesian factorization system.
Namely, the fiberwise dual $\Ca^{\vee}$ of a category $\Ca$ equipped with a
cartesian factorization system is the category of spans whose left leg is in the
left class and whose right leg is in the right class. We show that $\Ca^{\vee}$
may be equipped with a cartesian factorization system such that $\Ca^{\vee\vee}
\simeq \Ca$. This gives an involution $(-)^{\vee} : \type{Cart} \to
\type{Cart}$ on the category of cartesian factorization systems (it does not
behave well with natural transformations, although we can rectify this issue by
passing to double categories, as in Theorem \ref{thm:double.functoriality}).
This dual was constructed for $\infty$-categories in \cite{barwick_glasman_nardin_2014}.

We will show in Theorem \ref{thm:fiberwise.op.square} that $(-)^{\vee}$ corresponds to the fiberwise opposite in the
sense that the following square of functors commutes (up to equivalence):
\[
\begin{tikzcd}
\type{Groth} \arrow[d, "(-)^{\text{fibop}}"'] \arrow[r, "\Phi"] & \type{Cart} \arrow[d, "(-)^{\vee}"] \\
\type{Groth} \arrow[r, "\Phi"']         & \type{Cart}            
\end{tikzcd}
\]
where $(-)^{\text{fibop}} : \type{Groth} \to \type{Groth}$ is the
fiberwise opposite of a Grothendieck fibration, constructed by passing through
indexed categories. The variance of these functors is rather tricky to extend to the level
of 2-categories; we only consider their action on morphisms and not natural
transformations here.

We then show that the natural double category one might construct out of $\Ea$
and $\Ea^{\vee}$ (in the usual way one makes a double category out of spans) is
equivalent to a \emph{Grothendieck double construction} of the indexed category
$\Ea_{(-)} : \Ba\op \to \Ca$ whose horizontal
category is the Grothendieck construction of $\Ea_{(-)}$ and whose vertical
category is the Grothendieck construction of the pointwise opposite. We will
show in Theorem \ref{thm:double.functoriality} that this construction is
2-functorial, landing in the 2-category of double categories, functors, and
horizontal transformations. This
construction plays a role in a double categorical theory of open dynamical
systems \cite{jaz2020dynamical}.

\begin{acknowledgements}
  The author would like to thank Emily Riehl for her careful reading of a draft
  of this paper, and for her helpful comments. 
The author would also like to thank Jonas
  Frey, Fosco Loregian, Richard Garner, Edoardo Lanari, and Charles Barwick for
  bringing the references of previous work to my attention.
\end{acknowledgements}
\section{Orthogonal Factorization Systems}
First, we recall the definition of an orthogonal factorization system, and a few
elementary lemmas. For proofs, see \cite{riehlfactorization}.
\begin{defn}
A \emph{orthogonal factorization system} $(\epi, \mono)$ on a category $\Ca$ is a
pair of collections of arrows $\bullet \epi \bullet$ and $\bullet \mono \bullet$ such that
\begin{enumerate}
\item $\epi$ and $\mono$ both contain all isomorphisms and are closed under composition.
\item Every map $f : A \to B$ factors as $f = me$ with $e \in \epi$ and $m \in
  \mono$.
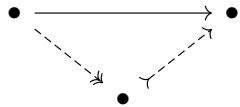
  \[
\begin{tikzcd}
\bullet \arrow[rr] \arrow[rd, two heads, dashed] &                                  & \bullet \\
                                                 & \bullet \arrow[ru, dashed, tail] &        
\end{tikzcd}
  \]
\item This factorization is \emph{uniquely functorial} in the sense that
  whenever we have a solid diagram:
  
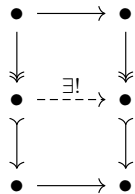
\[
\begin{tikzcd}
\bullet \arrow[r] \arrow[d, two heads]                & \bullet \arrow[d, two heads] \\
\bullet \arrow[d, tail] \arrow[r, "\exists!", dashed] & \bullet \arrow[d, tail]      \\
\bullet \arrow[r]                                     & \bullet                     
\end{tikzcd}
  \]
  We have a unique dashed arrow making the diagram commute.
\end{enumerate}
\end{defn}
We recall that this definition implies that $\La$ is \emph{orthogonal} to $\Ra$
in the sense that every lifting problem:
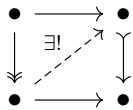
\[
\begin{tikzcd}
\bullet \arrow[r] \arrow[d, two heads]           & \bullet \arrow[d, tail] \\
\bullet \arrow[r] \arrow[ru, "\exists!", dashed] & \bullet                
\end{tikzcd}
\]
has a \emph{unique} solution, given by the dashed arrow. We now recall a few
elementary lemmas in the theory of orthogonal factorization systems. 
\begin{lem}[Saturation]\label{lem:saturation}
Let $(\epi, \mono)$ be an orthogonal factorization system. Then a map is in
$\epi$ (resp. $\mono$) if and only if it satisfies the unique left (resp. right)
lifting property against $\mono$s (resp. $\epi$s).
\end{lem}

\begin{lem}[Cancellation]
Let $(\epi, \mono)$ be an orthogonal factorization system. For any composable
arrows $f$ and $g$:
\begin{itemize}
\item if $f$ and $gf$ are in $\epi$, then so is $g$, and
  \item if $g$ and $gf$ are in $\mono$, then so is $f$.
\end{itemize}
\end{lem}

\begin{lem}
Let $(\epi, \mono)$ be an orthogonal factorization system. Then the right class
$\mono$ is preserved under pullback.
\end{lem}

\begin{lem}
  Let $(\epi, \mono)$ be an orthogonal factorization system. Then a map is an
  isomorphism if and only if it is in both $\epi$ and $\mono$.
\end{lem}

\section{Cartesian Factorization Systems}

Now we come to our main definition.
\begin{defn}\label{defn:cartesian.factorization.system}
  An orthogonal factorization system $(\epi, \mono)$ is \emph{cartesian} if
  \begin{enumerate}
    \item (Left 2-of-3) The class $\epi$ satisfies $2$-of-$3$: if $g$ and $gf$ are in $\epi$,
      then so is $f$.
    \item (Right Stability) Pullbacks of $\epi$s along $\mono$s exist and are in $\epi$.
  \end{enumerate}
  We will refer to an orthogonal factorization system that satisfies (2) ---
  Right Stability --- as a right stable orthogonal factorization system.
  
  We define the 2-category $\type{Cart}$ of cartesian factorization systems to
  consist of categories with cartesian factorization systems, functors which
  preserve both classes, and natural transformations between such functors. 
\end{defn}

In this section, we will see a few basic properties of cartesian factorization
systems, and then relate them to Grothendieck fibrations. We will end the
section by seeing that certain cartesian factorization systems are equivalent to
Grothendieck fibrations with a right adjoint right inverse. 

\begin{lem}\label{lem:cartesian.fact.pullback.square}
  Let $(\epi, \mono)$ be a cartesian factorization system on a category $\Ca$.
  Then every square of the form
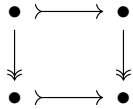
\[
\begin{tikzcd}
\bullet \arrow[d, two heads] \arrow[r, tail]  & \bullet \arrow[d, two heads] \\
\bullet \arrow[r, tail]                                 & \bullet                     
\end{tikzcd}
\]
is a pullback.
\end{lem}
\begin{proof}
  Consider the comparison map to the pullback:
  
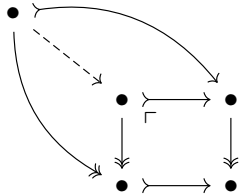
\[
\begin{tikzcd}
\bullet \arrow[rdd, two heads, bend right] \arrow[rrd, tail, bend left] \arrow[rd, dashed] &                                                         &                              \\
                                                                                           & \bullet \arrow[d, two heads] \arrow[r, tail] \arrow[rd, phantom, "\ulcorner" very near start] & \bullet \arrow[d, two heads] \\
                                                                                           & \bullet \arrow[r, tail]                                 & \bullet                     
\end{tikzcd}
  \]
  The comparison map is in $\mono$ by cancelation for the right class, and it is
  is in $\epi$ because $\epi$ satisfies 2-of-3 by hypothesis. Therefore, it is
  an isomorphism, and the outer square is a pullback.
\end{proof}

\begin{rmk}
  A \emph{modality} is an orthogonal factorization system on a category with
  finite limits in which the left class is stable under pullback. Clearly, a
  modality is in particular a right stable orthogonal factorization system. A
  modality is \emph{lex} if and only if the left class satisfies 2-of-3; that
  is, if and only if it is a cartesian factorization system, in addition to a modality.
\end{rmk}

This definition is meant to abstract some crucial features of the
vertical/cartesian orthogonal factorization system associated to a Grothendieck fibration.

\begin{prop}\label{prop:groth.fib.cartesian.fact}
  Let $p : \Ea \to \Ba$ be a Grothendieck fibration. The induced
  vertical/cartesian orthogonal factorization system $(\epi, \mono)$ is
  cartesian. This gives a 2-functor $\Phi : \type{Groth} \to \type{Cart}$ from the
  2-category of Grothendieck fibrations to the 2-category of cartesian
  factorization systems.
\end{prop}
\begin{proof}
The vertical morphisms satisfy 2-of-3 since they are, by definition, those
morphisms sent by $p$ to isomorphisms. It remains to show, then, that pullbacks
of vertical maps over cartesian maps exist and are vertical. Suppose we have a
diagram as follows:
\[
\begin{tikzcd}
                        & B \arrow[d, "g", two heads] \\
A \arrow[r, "f"', tail] & C                          
\end{tikzcd}
\]
We may complete this into a square:
\[
\begin{tikzcd}
(pf)^{\ast}B \arrow[r, tail, "k"] \arrow[d, two heads, dashed, "j"] & B \arrow[d, "g", two heads] \\
A \arrow[r, "f"', tail]                                   & C                          
\end{tikzcd}
\]
We do this in the usual way that one defines the base change operation on
fibers: namely, we take $k$ to be a cartesian lift of $(pg)\inv (pf)$ and $j$ is
the unique morphism determined by the universal property of $f$. We note that
$pj = \id_{pA}$, so that it is vertical. It remains to
show that this square is a pullback.

Given a solid diagram like so:
\[
\begin{tikzcd}
Z \arrow[rdd, bend right] \arrow[rrd, bend left] \arrow[rd, dashed] &                                                                      &                             \\
                                                                    & (pf)^{\ast}B \arrow[r, "k", tail] \arrow[d, "j"', two heads, dashed] & B \arrow[d, "g", two heads] \\
                                                                    & A \arrow[r, "f"', tail]                                              & C                          
\end{tikzcd}
\]
The dashed arrow exists uniquely by the universal property of $k$, using $(pj)\inv$ to
create the triangle in $\Ba$ to lift. 

We can quickly show that this assembles into a 2-functor. Given a cartesian
functor $p \to q$:
\[
\begin{tikzcd}
\Ea \arrow[d, "p"'] \arrow[r, "F"] & \Ea' \arrow[d, "q"] \\
\Ba \arrow[r, "G"']                & \Ba'               
\end{tikzcd}
\]
$F$ preserves cartesian morphisms by hypothesis, and we can see that $F$ preserves vertical morphisms by the commutativity, up to
isomorphism, of this square.
\end{proof}

There is, in fact, a strong relation between cartesian factorization systems and
Grothendieck fibrations. We will now explore this relationship, beginning with a
few preliminary definitions and lemmas.
\begin{defn}\label{defn:enough.injectives}
Let $(\epi, \mono)$ be an orthogonal factorization system. An object $X$ is
\emph{$\epi$-injective} (or just, \emph{injective}) if every extension problem like so:
\[
\begin{tikzcd}
\bullet \arrow[r] \arrow[d, two heads]  & X \\
\bullet \arrow[ru, "\exists!"', dashed] &  
\end{tikzcd}
\]
admits a unique solution. We say that $(\epi, \mono)$ \emph{has enough $\epi$-injectives} if for every $X$ there is a map $X \epi Y$ (in $\epi$) into an
injective object $Y$.
\end{defn}

If our category $\Ca$ has a terminal object, then this theory of injective
objects becomes particularly simple.
\begin{lem}
If $\Ca$ has a terminal object, then an object $X$ is injective if and only if
the terminal map $X \mono \ast$ is in the right class.
\end{lem}

\begin{rmk}
A modality, which is a stable orthogonal factorization system on a category with
finite limits, has enough $\epi$-injectives given by factoring the terminal
morphism. In this case, we call the operation of injective replacement the
\emph{modal operator}.
\end{rmk}

However, all of the basic facts which follow immediately from this
characterization of injective objects --- that every map between them is in the
right class $\mono$, that if $Y$ is injective and $X \mono Y$ then so is $X$ ---
follow as well without the presence of a terminal object.
\begin{lem}\label{lem:maps.between.injective.right}
Every map between injective objects is in the right class $\mono$.
\end{lem}
\begin{proof}
We show that if $X$ and $Y$ are injective and $f : X \to Y$, then every lifting
problem
\begin{equation}\label{diag:injective.lifting1}
\begin{tikzcd}
A \arrow[r] \arrow[d, two heads]           & X \arrow[d, "f"] \\
B \arrow[r] \arrow[ru, "\exists!", dashed] & Y               
\end{tikzcd}
\end{equation}
admits a unique solution, so that $f$ is in $\mono$ by the saturation of that
class. Since $X$ is injective, there is a unique solution to the following
extension problem:
\begin{equation}\label{diag:injective.lifting2}
\begin{tikzcd}
A \arrow[r] \arrow[d, two heads] & X \\
B \arrow[ru, "\exists!", dashed] &  
\end{tikzcd}
\end{equation}
We just need to show that this solves the above lifting problem. Since $Y$ is
also injective, there is a unique solution to the following extension problem
which is already given by the map $B \to Y$:
\[
\begin{tikzcd}
A \arrow[r] \arrow[d, two heads]             & X \arrow[r, "f"] & Y \\
B \arrow[ru, "\exists!", dashed] \arrow[rru] &                  &  
\end{tikzcd}
\]
therefore the composite $B \to X \to Y$ is equal to the map $B \to Y$.
Uniqueness of this result follows because, as we just saw, any solution of the
extension problem
(\ref{diag:injective.lifting2}) gives a solution of the lifting problem (\ref{diag:injective.lifting1}).
\end{proof}

\begin{lem}\label{lem:injective.closed.under.right}
If $Y$ is injective and $k : X \mono Y$ is a map in the right class, then $X$ is injective. 
\end{lem}
\begin{proof}
Consider an extension problem of the following form, for which we will produce a unique
solution:
\[
\begin{tikzcd}
A \arrow[r, "g"] \arrow[d, "f"', two heads] & X \\
B                                           &  
\end{tikzcd}
\]
Now since $Y$ is injective, we have a unique solution to this extension problem:
\[
\begin{tikzcd}
\bullet \arrow[r, "g"] \arrow[d, "f"', two heads] & X \arrow[r, "k", tail] & Y \\
\bullet \arrow[rru, "\exists!"', dashed]          &                        &  
\end{tikzcd}
\]
Re-arranging, we get a square
\[
\begin{tikzcd}
\bullet \arrow[r, "g"] \arrow[d, "f"', two heads]                     & X \arrow[d, "k", tail] \\
\bullet \arrow[r, "\exists!"', dashed] \arrow[ru, "\exists!", dashed] & Y                     
\end{tikzcd}
\]
which has a unique filler by orthogonality. Therefore, we have a unique solution
to the extension problem we started with.
\end{proof}

\begin{prop}
Let $\Ca_R$ denote the full subcategory of $\epi$-injective objects in an orthogonal
factorization system $(\epi, \mono)$. Then $(\epi, \mono)$ has enough injectives if and only if the inclusion $\Ca_R \hookrightarrow \Ca$ has a left
adjoint.
\end{prop}
\begin{proof}
  Suppose that $(\epi, \mono)$ has enough injectives. We will show that the slice category $X
  \down \Ca_R$ of $X$ over the injective objects admits an initial object.
  Suppose that $X \epi Y$ is an injective replacement of $X$. Then the unique
  extension condition of injective objects says that this is initial in $X \down
  \Ca_R$. 

  Now suppose that $X \down \Ca_R$ has an initial object $X \to Y$. We will show
  that this is in the left class, and is therefore an injective replacement of
  $X$. Let $X \epi M \mono Y$ be the factorization of $X \to Y$, seeking to show
  that $M \mono Y$ is an isomorphism. Note that
  by Lemma \ref{lem:injective.closed.under.right}, $M$ is also injective.
  Therefore, by the initiality of $X \to Y$, we have a unique map $Y \to M$
  under $X$. Since $X \to Y$ is initial, the induced endomorphism $Y \to M
  \mono Y$ is the identity. On the other hand, by cancellation for $\epi$, the
  endomorphism $M \mono Y \to M$ is in $\epi$; but by Lemma
  \ref{lem:maps.between.injective.right} maps between injective objects are in
  $\mono$, and therefore this endomorphism is an isomorphism. It follows that $M
  \mono Y$ is an isomorphism as well. 
\end{proof}

\begin{prop}\label{prop:cartesian.fact.groth.fib}
Suppose that $(\epi, \mono)$ is a right stable orthogonal factorization system
on $\Ca$ which has enough injectives. Then the injective replacement functor $R : \Ca \to
\Ca_R$ is a Grothendieck fibration.

Furthermore, $(\epi, \mono)$ is cartesian if
and only if and $\epi$ is the class of
vertical morphisms and $\mono$ the class of cartesian morphisms for this
Grothendieck fibration.
\end{prop}
\begin{proof}
We show that every map in $\Ca_R$ has a cartesian lift. That is, suppose we have
a diagram like so:
\[
\begin{tikzcd}
                        & X \arrow[d, "e", two heads] \\
A \arrow[r, "m"', tail] & B                          
\end{tikzcd}
\]
Define the lift of the map $A \mono B$ to be the pullback:
\[
\begin{tikzcd}
X' \arrow[r, "e^{\ast} m", tail] \arrow[d, two heads] \arrow[rd, phantom, "\ulcorner" very near start] & X \arrow[d, "e", two heads] \\
A \arrow[r, "m"', tail]                                          & B                          
\end{tikzcd}
\]
That this is cartesian follows quickly from the universal property of the pullback:
\[
\begin{tikzcd}
Z \arrow[d, two heads] \arrow[rrd, bend left] \arrow[rd, dashed] &                                                                  &                             \\
RZ \arrow[rd, tail, bend right] \arrow[rrd]                      & X' \arrow[r, "e^{\ast} m", tail] \arrow[d, crossing over, two heads] \arrow[rd, phantom, "\ulcorner" very near start] & X \arrow[d, "e", two heads] \\
                                                                 & A \arrow[r, "m"', tail]                                          & B                          
\end{tikzcd}
\]
Given a triangle with vertex $RZ$, we get the above solid diagram and so a
unique dashed map.

Now, we show that the induced vertical/cartesian factorization system is
$(\epi, \mono)$ if and only if the factorization system is cartesian. Now, if
$\epi$ is the vertical class, then it satisfies 2-of-3. So, suppose that $\epi$
satisfies 2-of-3, and we will show that a map is in $\epi$ if and only if it is sent to an
isomorphism by $R$. Let $f : X \epi Y$, and consider the following naturality
square for $R$:
\[
\begin{tikzcd}
X \arrow[d, two heads] \arrow[r, two heads] & Y \arrow[d, two heads] \\
RX \arrow[r, tail]                          & RY                    
\end{tikzcd}
\]
By cancellation, $Rf : RX \mono RY$ is also in $\epi$, and is therefore an
isomorphism. Similarly, if $Rf$ is an isomorphism for arbitrary $f : X \to Y$,
then by 2-of-3, $f$ is in $\epi$. 
\end{proof}

\begin{rmk}
  Working in homotopy type theory, Rijke and Cherubini \cite{cherubini2020modal}
  have associated to every modality another modal-equivalence/modal-{\'e}tale factorization system. This
  factorization system is the vertical/cartesian factorization system induced by
  the modal operator $! : \Type \to \Type_!$.
\end{rmk}

\begin{thm}\label{thm:rsofs.to.groth.pseudofunctor}
Sending a right stable orthogonal factorization system with enough $\epi$-injectives to the Grothendieck fibration $R : \Ca \to \Ca_R$ given by
injective replacement constitutes a pseudo-functor $\Xi : \type{RSOFS}_R \to
\type{Groth}_{R}$ from the 2-category of right stable orthogonal factorization
systems with enough injectives and functors preserving both classes to the
2-category of Grothendieck fibrations with right adjoint right inverses and
cartesian functors between them.
\end{thm}
\begin{proof}
  
We begin by construction $\Xi$. Given a
cartesian factorization system $(\epi, \mono)$ on a category $\Ca$ enough injectives, we
send it to the injective replacement $R : \Ca \to \Ca_R$, which is a
Grothendieck fibration with a right adjoint right inverse by Proposition
\ref{prop:cartesian.fact.groth.fib}. Given a functor $F : \Ca \to \Da$ which
preserves the left class, take the square
\[
\begin{tikzcd}
\Ca \arrow[d, "R"'] \arrow[r, "F"] & \Da \arrow[d, "R"] \\
\Ca_R \arrow[r, "RF"']             & \Da_R             
\end{tikzcd}
\]
We can see that this square commutes up to isomorphism by considering the
injective replacement $X \epi RX$; its image $FX \epi FRX$ gives rise to an
isomorphism $RFX \xto{\sim} RFRX$. Given a natural transformation $\theta : F
\to G$ between such functors, we can take $R\theta : RF \to RG$ to get an
appropriate 2-cell in $\type{Groth}$.

We now need to show the pseudofunctoriality of this construction. We define
\begin{itemize}
  \item the unitor $u_{\Ca} : \id_{\Ca_{R}} \xto{\sim} R$ given by noting any
      injective replacement $X \epi RX$ is also in $\mono$ and therefor is an isomorphism.
  \item the compositor $f_{F, G} : RGRF \xto{\sim} RGF$ given at $X$ by taking
    the injective replacement $r_{FX} : FX \epi RFX$ and applying $RG$ to it to
    get an isomorphism $RGr_{FX} : RGFX \xto{\sim} RGRFX$, and taking the
    inverse
    $$f_{F,G} := (RGr_{FX})\inv.$$
\end{itemize}
The coherence conditions follow from the uniqueness of injective replacements.
\end{proof}

\begin{thm}\label{thm:cartesian.equivalence}
The pseudo-functors $\Phi$ and $\Xi$ of Proposition
\ref{prop:groth.fib.cartesian.fact} and Theorem
\ref{thm:rsofs.to.groth.pseudofunctor} assemble into an equivalence:
\[
\begin{tikzcd}
\type{Groth}_{R} \arrow[rr, "\Phi", bend left] & \simeq & \type{Cart}_R \arrow[ll, "\Xi", bend left]
\end{tikzcd}
\]
\end{thm}
\begin{proof}
First, we should show that if a Grothendieck fibrations $p : \Ea \to \Ba$ admits
a right adjoint right inverse $r$, then $\Ea$ admits vertical replacements. 

For the unit, we take the identity $\Ca \to \Phi\Xi \Ca$, which preserves both
classes by Proposition \ref{prop:cartesian.fact.groth.fib}. For the counit, we
take the cartesian functor
\[
\begin{tikzcd}
\Ea \arrow[r, equals] \arrow[d, "rp"'] & \Ea \arrow[d, "p"] \\
\Ea_R \arrow[r, "p"']          & \Ba               
\end{tikzcd}
\]
We note that we may take $rp : \Ea \to \Ea_R$ for the injective replacement
where $r : \Ba \to \Ea$ is the right adjoint to $p$
since $rB$ is $\epi$-injective for all $B \in \Ba$. This square commutes up to
natural isomorphism since, by hypothesis, the unit $\eta : X \to rp X$ is
vertical, and therefore $p \eta : p X \xto{\sim} prp X$ gives a witness to
comutativity.
It remains to show that this counit is invertible. We will show that $p : \Ea_R
\to \Ba$ is an equivalence with inverse $r : \Ba \to \Ea_R$. We note that the
unit $\eta : X \to rpX$ is an isomorphism since as a map between
$\epi$-injectives, it is in $\mono$ and it is by hypothesis in $\epi$. The
counit $\epsilon : prX \to X$ is the identity by assumption. 
\end{proof}

\section{The Fiberwise Dual}

As we saw in the last section, there is a close relationship between cartesian
factorization systems and Grothendieck fibrations. In this section, we will give
a construction which corresponds to taking the \emph{fiberwise opposite} of a
indexed category.

\begin{defn}
Let $(\epi, \mono)$ be a cartesian factorization system on a category
$\Ca$. Define the \emph{fiberwise dual} of $\Ca$ to be the category $\Ca^{\vee}$ with the
same objects as $\Ca$, but where a morphism $A \to B$ is a span of the following form:
\[
\begin{tikzcd}
  & \bullet \arrow[ld, two heads] \arrow[rd, tail] &   \\
A &                                                & B
\end{tikzcd}
\]
Composition is given by pullback, as is usual with spans. We also define the
dual classes $\epi^{\vee}$ and $\mono^{\vee}$ to be the spans of the following
forms respectively:
\[
\begin{tikzcd}
        & \bullet \arrow[ld, two heads] \arrow[rd, equals] &         &  &         & \bullet \arrow[ld, equals] \arrow[rd, tail] &         \\
\bullet &                                          & \bullet &  & \bullet &                                     & \bullet
\end{tikzcd}
\]
\end{defn}

Though this is naturally a bicategory, there is at most one morphism between any
two such spans and it must be an isomorphism. Consider a
morphism between these spans as follows:
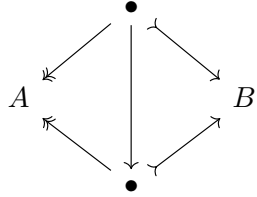
\[
\begin{tikzcd}
  & \bullet \arrow[ld, two heads] \arrow[rd, tail] \arrow[dd] &   \\
A &                                                           & B \\
  & \bullet \arrow[lu, two heads] \arrow[ru, tail]            &  
\end{tikzcd}
\]
By cancellation for $\mono$, the middle arrow is in $\mono$; by 2-of-3 for $\epi$, it
is also in $\epi$. But it is therefore an isomorphism. It is unique by
the following lemma:
\begin{lem}\label{lem:span.morphism.uniqueness}
Let $(\epi, \mono)$ be a cartesian factorization system. Then given a solid
diagram as below, there may be at most one dashed arrow making the diagram
commute:
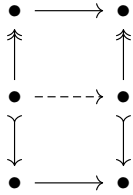
\[
\begin{tikzcd}
\bullet \arrow[r]                                              & \bullet                                      \\
\bullet \arrow[u, two heads] \arrow[d, tail] \arrow[r, dashed] & \bullet \arrow[u, two heads] \arrow[d, tail] \\
\bullet \arrow[r]                                              & \bullet                                     
\end{tikzcd}
\]
\end{lem}
\begin{proof}
We will show that any two $(\epi, \mono)$-factorizations of such a dashed arrow
will agree. Consider a factorization of the top square and note that by unique
functoriality, there is a unique dotted arrow filling the diagram
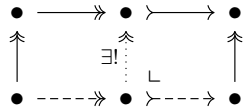
\[
\begin{tikzcd}
\bullet \arrow[r, two heads]                              & \bullet \arrow[r, tail]                                                  & \bullet                      \\
\bullet \arrow[u, two heads] \arrow[r, two heads, dashed] & \bullet \arrow[r,
dashed, tail] \arrow[u, "\exists!", two heads, dotted] \arrow[ur, phantom,
"\llcorner" very near start]& \bullet \arrow[u, two heads]
\end{tikzcd}
\]
By Lemma \ref{lem:cartesian.fact.pullback.square}, the right square is a
pullback. Therefore, the $(\epi, \mono)$-factorization of any dashed arrow in
the original diagram is given by the universal property of this pullback, and is
therefore uniquely determined by the data of the solid diagram.
\end{proof}

Therefore, the natural map from the bicategory of such spans to the locally
discrete bicategory given by taking the category of isomorphism classes of spans
and reincluding as a bicategory is an equivalence. We are therefore justified in
thinking of $\Ca^{\vee}$ as a category.
\begin{rmk}
In univalent foundations, $\Ca^{\vee}$ would be \emph{proven} to be a category,
so long as $\Ca$ is a (univalent) category. That is, showing that there is at
most a single isomorphism between any two such spans shows that the type of such
spans is in fact a set.
\end{rmk}

\begin{thm}
For $(\epi, \mono)$ a cartesian factorization system on a category $\Ca$, the
dual classes $(\epi^{\vee}, \mono^{\vee})$ form a cartesian factorization
system on $\Ca^{\vee}$.
\end{thm}
\begin{proof}
  It is evident that every morphism in $\Ca^{\vee}$ factors as a morphism in
  $\epi^{\vee}$ followed by a morphism in $\mono^{\vee}$. We will now show the
  unique functoriality condition. Given a commuting square in $\Ca^{\vee}$ like
  so:
  
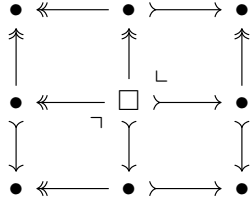
\[
\begin{tikzcd}
\bullet                                      & \bullet \arrow[l, two heads] \arrow[r, tail]                          & \bullet                                      \\
\bullet \arrow[d, tail] \arrow[u, two heads] & \square \arrow[l, two heads]
\arrow[u, two heads] \arrow[r, tail] \arrow[d, tail] \arrow[ld, phantom, "\urcorner" very near start] \arrow[ru, phantom, "\llcorner" very near start] & \bullet \arrow[u, two heads] \arrow[d, tail] \\
\bullet                                      & \bullet \arrow[r, tail] \arrow[l, two heads]                          & \bullet                                     
\end{tikzcd}
  \]
  We note that the middle arrows are in the classes described above by pullback
  preservation. We may therefore expand this diagram to show a functorial
  factorization. 
  
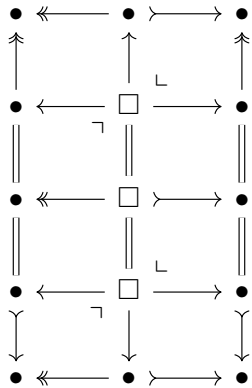
\[
\begin{tikzcd}
\bullet                                & \bullet \arrow[l, two heads] \arrow[r, tail]                          & \bullet                                \\
\bullet \arrow[u, two heads] \arrow[d, equals] & \square \arrow[d, equals]
\arrow[u] \arrow[l] \arrow[r] \arrow[ld, phantom, "\urcorner" very near start]
\arrow[ru, phantom, "\llcorner" very near start] & \bullet \arrow[u, two heads]
\arrow[d, equals] \\
\bullet                                & \square \arrow[l, two heads] \arrow[r, tail]                          & \bullet                                \\
\bullet \arrow[u, equals] \arrow[d, tail]      & \square \arrow[l] \arrow[d]
\arrow[r] \arrow[u, equals] \arrow[ru, phantom, "\llcorner" very near start]
\arrow[ld, phantom, "\urcorner" very near start] & \bullet \arrow[u, equals] \arrow[d, tail]      \\
\bullet                                & \bullet \arrow[r, tail] \arrow[l, two heads]                          & \bullet                               
\end{tikzcd}
  \] 
  Since these diagrams are merely re-arrangements of each other, we find that
  this functoriality is unique.

  Now, $\epi^{\vee}$ satisfies 2-of-3 since $\epi$ does and $\epi^{\vee} =
  \epi\op$ as categories. It remains then to show that pullbacks of
  $\epi^{\vee}$s along $\mono^{\vee}$s exist and are in $\epi^{\vee}$. Suppose
  we have such a cospan:
  \[
\begin{tikzcd}
  &                                   & B                                      \\
  &                                   & C \arrow[d, equals] \arrow[u, two heads] \\
A & A \arrow[l, equals] \arrow[r, tail] & C                                     
\end{tikzcd}
  \]
  We see that this data consists of morphisms $A \mono C \epi B$. Let $A \epi D
  \mono B$ be the factorization of the composite of these morphisms. Then we may
  form the following square:
  \[
\begin{tikzcd}
D                                & D \arrow[l, equals] \arrow[r, tail]                                                      & B                                \\
A \arrow[d, equals] \arrow[u, two heads] & A \arrow[u, two heads] \arrow[r, tail] \arrow[d, equals] \arrow[l, equals] \arrow[ld, phantom, "\urcorner" very near start] \arrow[ru, phantom, "\llcorner" very near start] & C \arrow[d, equals] \arrow[u, two heads] \\
A                                & A \arrow[l, equals] \arrow[r, tail]                                                      & C                               
\end{tikzcd}
  \]
  The top left square is a pullback by Lemma
  \ref{lem:cartesian.fact.pullback.square}, so this is a commuting square in
  $\Ca^{\vee}$ . Now, suppose we have a square as follows:
  \[
\begin{tikzcd}
X                                                 & Y \arrow[l, two heads] \arrow[r, tail]                                                                    & B                                \\
Z \arrow[d, "\alpha"', tail] \arrow[u, two heads] & Z \arrow[u, "\beta", two
heads] \arrow[r, tail] \arrow[d, "\alpha", tail] \arrow[l, equals] \arrow[ld, phantom,
"\urcorner" very near start] \arrow[ru,
phantom, "\llcorner" very near start] & C \arrow[d, equals] \arrow[u, two heads] \\
A                                                 & A \arrow[l, equals] \arrow[r, tail]                                                                               & C                               
\end{tikzcd}
  \]
  The data of the terminal map into the limit cone consists of a
  diagram:
  \[
\begin{tikzcd}
  &                                                                                                              & X                                              &                                                                                       &   \\
  & Z \arrow[ld, tail] \arrow[ru, two heads] \arrow[d, dashed]
  \arrow[r, dashed] \arrow[rd, phantom, "\ulcorner" very near start]
  & ? \arrow[d, dashed, tail] \arrow[u, dashed, two heads] & Y \arrow[rd, tail]
  \arrow[lu, two heads] \arrow[l, dashed] \arrow[d, dashed] \arrow[ld,
  phantom, "\urcorner" very near start] &   \\
A & A \arrow[r, two heads] \arrow[l, equals]
& D                                              & D \arrow[r, tail] \arrow[l,
equals]                                                           & B
\end{tikzcd}
\]

Many of our choices are fixed by commutativity and the pullback conditions. We
end up with a diagram like this:
  \[
\begin{tikzcd}
  &                                                                                                              & X                                              &                                                                                       &   \\
  & Z \arrow[ld, tail] \arrow[ru, two heads] \arrow[d, "\alpha"', tail]
  \arrow[r, "\beta", two heads] \arrow[rd, phantom, "\ulcorner" very near start]
  & Y \arrow[d, dashed, tail] \arrow[u, two heads] & Y \arrow[rd, tail]
  \arrow[lu, two heads] \arrow[l, equals] \arrow[d, dashed, tail] \arrow[ld,
  phantom, "\urcorner" very near start] &   \\
A & A \arrow[r, two heads] \arrow[l, equals]
& D                                              & D \arrow[r, tail] \arrow[l,
equals]                                                           & B
\end{tikzcd}
  \]
  which ultimately depends only on the data of the dashed map
  $\begin{tikzcd} Y \arrow[r, tail, dashed] & D\end{tikzcd}$ making
  the square on the left and the triangle on the right commute. There is a small point to focus on here: the
  choice of $\beta : Z \epi Y$. Since the square will be a pullback, we need a
  map $Z \epi Y$ such that the composite $Z \epi Y \mono D$ is equal to $Z \mono
  A \epi D$; but such factorizations are unique, so if we can provide such a
  composite with left component $\beta : Z \epi Y$, then every choice will be
  isomorphic to this choice.

  Luckily, asking that the square on the left and the triangle on the right
  commute can be arranged into the single square:
  \[
\begin{tikzcd}
Z \arrow[r, tail] \arrow[d, two heads]             & A \arrow[r, two heads] & D \arrow[d, tail] \\
Y \arrow[rr, tail] \arrow[rru, "\exists!", dashed] &                        & B                
\end{tikzcd}
  \]
  This square has a unique filler by orthogonality, which is in the right class
  by cancellation.
  
\end{proof}

\begin{prop}
For any cartesian factorization system $(\epi, \mono)$ on a category $\Ca$, the
functor $\Ca \to \Ca^{\vee\vee}$ sending a map $f : X \to Y$ which factors as $X
\epi M \mono Y$ to the span-of-spans
\[
\begin{tikzcd}
  & X \arrow[ld, equals] \arrow[rd, two heads] &   & M \arrow[ld, equals] \arrow[rd, tail] &   \\
X &                                    & M &                               & Y
\end{tikzcd}
\]
is an equivalence in $\type{Cart}$.
\end{prop}
\begin{proof}
We will show that this is a functor, since it is immediately seen to be fully
faithful and essentially surjective. It quite clearly preserves both classes.

Identities are sent to identities since $X \xto{\id} X \xto{\id} X$ is a $(\epi,
\mono)$-factorization of $\id$. As for preserving compostion, we note that
composition in $\Ca^{\vee\vee}$ is given by taking the pullback in $\Ca^{\vee}$
which is given by factoring in $\Ca$. For any composite $X \xto{f} Y \xto{g} Z$, its
factorization is given by the factorization of $M \mono Y \epi N$ where $X \epi
M \mono Y$ and $Y \epi N \mono Z$ are the factorizations of $f$ and $g$
respectively in the following way:
\[
\begin{tikzcd}
X \arrow[rr, "f"] \arrow[rd, two heads] \arrow[rrdd, two heads, bend right] &                                          & Y \arrow[rr, "g"] \arrow[rd, two heads]           &                    & Z \\
                                                                            & M \arrow[ru, tail] \arrow[rd, two heads] &                                                   & N \arrow[ru, tail] &   \\
                                                                            &
                                                                            & K
                                                                            \arrow[ru,
                                                                            tail]
                                                                            \arrow[rruu,
                                                                            tail,
                                                                            bend
                                                                            right] &                    &  
\end{tikzcd}
\]

\end{proof}

\begin{prop}\label{prop:fiberwise.dual.functor}
The fiberwise dual construction $\Ca \mapsto \Ca^{\vee}$ gives a functor
$(-)^{\vee} : \type{Cart} \to \type{Cart}$.
\end{prop}
\begin{proof}
Given any functor $F : \Ca \to \Da$ which preserves both classes of the
cartesian factorization system on $\Ca$, we note that $F$ preserves any pullback
square of the form:
\[
\begin{tikzcd}
\bullet \arrow[r, tail] \arrow[d, two heads] & \bullet \arrow[d, two heads] \\
\bullet \arrow[r, tail]                      & \bullet                     
\end{tikzcd}
\]
by Lemma \ref{lem:cartesian.fact.pullback.square}. Therefore, we can define
$F^{\vee} : \Ca^{\vee} \to \Da^{\vee}$ by 
\[
\begin{tikzcd}
  & B \arrow[ld, "f"', two heads] \arrow[rd, "g", tail] &   &         &    & FB \arrow[ld, "Ff"', two heads] \arrow[rd, "Fg", tail] &    \\
A &                                                     & C & \mapsto & FA &                                                        & FC
\end{tikzcd}
\]
This is evidently functorial.
\end{proof}

We now show that the fiberwise dual $\Ca^{\vee}$ does in fact represent the
fiberwise dual of a Grothendieck fibration.
\begin{thm}\label{thm:fiberwise.op.square}
  The following square of functors commutes up to natural equivalence:
  \[
\begin{tikzcd}
\type{Groth} \arrow[d, "(-)^{\text{fibop}}"'] \arrow[r, "\Phi"] & \type{Cart} \arrow[d, "(-)^{\vee}"] \\
\type{Groth} \arrow[r, "\Phi"']         & \type{Cart}            
\end{tikzcd}
  \]
\end{thm}
\begin{proof}
We will show that if $p : \Ea \to \Ba$ is the Grothendieck construction of an
indexed category $\Ea_{(-)} : \Ba\op \to \Cat$, then $\Ea^{\vee} \simeq \int_{B
  : \Ba} \Ea_B\op$ is the Grothendieck construction of the fiberwise opposite.
First, let's set some notation:
\begin{itemize}
  \item We will denote objects of the Grothendieck constructions $\int_{B : \Ba}
    \Ea^{(\text{op})}$ by $\lens{E}{B}$ where $B \in \Ba$ and $E \in \Ea_B$.
  \item A morphism in the Grothendieck construction $\Ea = \int_{B : \Ba} \Ea_B$ will
    be denoted by
    $$\lens{f_{\sharp}}{f} : \lens{E}{B} \rightrightarrows \lens{E'}{B'}$$
    where $f : B \to B'$ and $f_{\sharp} : E \to f^{\ast}E'$.
  \item A morphism in the Grothendieck construction $\int_{B : \Ba} \Ea_{B}\op$
    of the pointwise opposite will be denoted by
    $$\lens{f^{\sharp}}{f} : \lens{E}{B} \leftrightarrows \lens{E'}{B'}$$
    where $f : B \to B'$ and $f^{\sharp} : f^{\ast}E' \to E$. We will refer to
    these morphisms as \emph{lenses}, following \cite{spivak2019generalized}.
\end{itemize}
We will construct a functor $\varphi : \int_{B : \Ba} \Ea_B\op \to \Ea^{\vee}$ as follows:
\begin{itemize}
  \item $\varphi$ acts as the identity on objects.
  \item A lens $\lens{f^{\sharp}}{f} : \lens{E}{B} \leftrightarrows
    \lens{E'}{B'}$ gets sent to the span:
\[
      \begin{tikzcd}
        & \lens{f^{\ast}E'}{B} \arrow[dr, equals, shift left, "\lens{\id}{f}"] \arrow[dr, shift right]
        \arrow[dl, equals, shift left] \arrow[dl, shift right, "\lens{f^{\sharp}}{\id}"']& \\
        \lens{E}{B} & & \lens{E'}{B'}
      \end{tikzcd}
\]
\end{itemize}
This clearly sends identities to identities and preserves both classes of the
cartesian factorization system; we need to show that it preserves
composition. Suppose we have $\lens{f^{\sharp}}{f} : \lens{E}{B}
\leftrightarrows \lens{E'}{B'}$ and $\lens{g^{\sharp}}{g} : \lens{E'}{B'}
\leftrightarrows \lens{E''}{B''}$ with composite $\lens{f^{\sharp} \circ
  f^{\ast}g^{\sharp}}{g \circ f}$. We may then form the following diagram in $\Ea$:
\[
      \begin{tikzcd}
        & & \lens{f^{\ast}g^{\ast}E''}{B} \arrow[dr, equals, shift left, "\lens{\id}{f}"] \arrow[dr, shift right]
        \arrow[dl, equals, shift left] \arrow[dl, shift right,
        "\lens{f^{\ast}g^{\sharp}}{\id}"'] \arrow[dd, phantom, "\mathbin{\rotatebox[origin=c]{-45}{$\ulcorner$}}" very near start]& & \\
        & \lens{f^{\ast}E'}{B} \arrow[dr, equals, shift left, "\lens{\id}{f}"] \arrow[dr, shift right]
        \arrow[dl, equals, shift left] \arrow[dl, shift right,
        "\lens{f^{\sharp}}{\id}"'] & & \lens{g^{\ast}E''}{B'} \arrow[dr, equals, shift left, "\lens{\id}{g}"] \arrow[dr, shift right]
        \arrow[dl, equals, shift left] \arrow[dl, shift right,
        "\lens{g^{\sharp}}{\id}"'] & \\
        \lens{E}{B} & & \lens{E'}{B'} & & \lens{E''}{B''}
      \end{tikzcd}
\]
The square is a pullback by Lemma \ref{lem:cartesian.fact.pullback.square}, and
the outer span is the span associated to the composite $\lens{f^{\sharp} \circ
  f^{\ast}g^{\sharp}}{g \circ f}$. 

It remains to show that $\varphi$ is an equivalence. It is clearly essentially
surjective; we will show it is fully faithful. Every
span with left leg vertical and right leg cartesian is of the form
\[
      \begin{tikzcd}
        & \lens{f^{\ast}E'}{B} \arrow[dr, equals, shift left, "\lens{\id}{f}"] \arrow[dr, shift right]
        \arrow[dl, equals, shift left] \arrow[dl, shift right, "\lens{f^{\sharp}}{\id}"']& \\
        \lens{E}{B} & & \lens{E'}{B'}
      \end{tikzcd}
\]
and is therefore involves precisely the data of a lens $\lens{f^{\sharp}}{f} :
\lens{E}{B} \leftrightarrows \lens{E'}{B'}$.
\end{proof}

Since the morphisms of $\Ca^{\vee}$ are certain sorts of spans, we can form a
double category whose vertical category is $\Ca^{\vee}$ and whose horizontal
category is $\Ca$ and where squares are the usual squares in the double category
of spans. This construction mimics, in terms of cartesian factorization systems,
a \emph{Grothendieck double construction} which produces a double category from
an indexed category.

\begin{defn}
Let $\Ea : \Ba\op \to \Cat$ be an indexed category, and let $f^{\ast}$ denote $\Ea(f)$. Its \emph{Grothendieck double construction} $\Ea \wr \Ea\op$ is the double category with:
\begin{itemize}
    \item Objects pairs $\lens{E}{B}$ with $B \in \Ba$ and $E \in \Ea(B)$.
    \item Vertical morphisms $\lens{f^{\sharp}}{f} : \lens{E}{B} \leftrightarrows \lens{E'}{B'}$ are morphisms in the Grothendieck construction $\int \Ea\op$ of the pointwise opposite of $\Ea$, namely pairs $f : B \to B'$ and $f^{\sharp} : f^{\ast}B' \to B$.
    \item Horizontal morphisms $\lens{g_{\sharp}}{g} : \lens{E}{B} \rightrightarrows \lens{E'}{B'}$ are morphisms in the Grothendick construction $\int \Ea$ of $\Ea$, namely pairs $g : B \to B'$ and $g_{\sharp} : E \to g^{\ast} E'$.
    \item There is a square
    \[
    \begin{tikzcd}
        \lens{E_1}{B_1} \arrow[r, shift left, "\lens{g_{1\sharp}}{g_1}"]\arrow[r, shift right] \arrow[d, leftarrow,  shift left] \arrow[d, shift right, "\lens{f_1^{\sharp}}{f_1}"'] & \lens{E_2}{B_2} \arrow[d, leftarrow, shift left, "\lens{f_2^{\sharp}}{f_2}"] \arrow[d, shift right] \\
        \lens{E_3}{B_3} \arrow[r, shift left]\arrow[r, shift right, "\lens{g_{2\sharp}}{g_2}"'] & \lens{E_4}{B_4}
    \end{tikzcd}
    \]
    if and only if the following diagrams commute:
    \begin{equation} \label{eqn:groth.double.diagram}
\begin{tikzcd}
B_1 \arrow[r, "g_1"] \arrow[d, "f_1"'] & B_2 \arrow[d, "f_2"] &  & f_1^{\ast}E_3 \arrow[r, "f_1^{\sharp}"] \arrow[d, "f_1^{\ast}g_{2\sharp}"'] & E_1 \arrow[dd, "g_{1\sharp}"] \\
B_3 \arrow[r, "g_2"']                  & B_4                  &  & f_1^{\ast}g_2^{\ast}E_4 \arrow[d, equals]                                           &                               \\
                                       &                      &  & g_1^{\ast}f_2^{\ast}E_4 \arrow[r, "g_1^{\ast}f_2^{\sharp}"']                & g_1^{\ast}E_2                
                                     \end{tikzcd}
                                     \end{equation}
We will call the squares in the Grothendieck double construction \emph{commuting squares}, since they represent the proposition that the ``lower'' and ``upper'' squares appearing in their boundary commute.
\end{itemize}

Composition is given as in the appropriate Grothendieck constructions. It just remains to show that commuting squares compose.
\begin{itemize}
    \item For vertical composition we appeal to the following diagram:
    \[
\begin{tikzcd}
f_1^{\ast}f_3^{\ast} E_5 \arrow[rr, "f_1^{\ast}f_3^{\sharp}"] \arrow[d, "f_1^{\ast}f_3^{\ast}g_{3\sharp}"'] &                                                                                 & f_1^{\ast}E_3 \arrow[r, "f_1^{\sharp}"] \arrow[d, "f_1^{\ast}g_{2\sharp}"] & E_1 \arrow[dd, "g_{1\sharp}"] \\
f_1^{\ast}f_3^{\ast}g_3^{\ast}E_6 \arrow[d,equals] \arrow[r,equals]                                                       & f_1^{\ast}g_2^{\ast}f_4^{\ast}E_6 \arrow[r, "f_1^{\ast}g_2^{\ast}f_4^{\sharp}"] & f_1^{\ast}g_2^{\ast}E_4 \arrow[d,equals]                                           &                               \\
g_1^{\ast}f_2^{\ast}f_4^{\ast}E_6 \arrow[rr, "g_1^{\ast}f_2^{\ast}f_4^{\sharp}"']                           &                                                                                 & g_1^{\ast}f_2^{\ast}E_4 \arrow[r, "g_1^{\ast}f_2^{\sharp}"']                & g_1^{\ast}E_2                
\end{tikzcd}
    \]
    The outer diagram is the ``upper'' square of the composite, while the ``upper'' squares of each factor appear in the top left and right respectively.
    \item For horizontal composition we appeal to the following diagram:
    \[
\begin{tikzcd}
f_1^{\ast}E_3 \arrow[rr, "f_1^{\sharp}"] \arrow[d, "f_1^{\ast}g_{2\sharp}"']     &                                                                                                           & E_1 \arrow[d, "g_{1\sharp}"]                      \\
f_1^{\ast}g_2^{\ast}E_4 \arrow[r,equals] \arrow[dd, "f_1^{\ast}g_2^{\ast}g_{4\sharp}"'] & g_1^{\ast}f_2^{\ast}E_4 \arrow[r, "g_1^{\ast}f_2^{\sharp}"] \arrow[d, "g_1^{\ast}f_2^{\ast}g_{4\sharp}"'] & g_1^{\ast}E_2 \arrow[dd, "g_1^{\ast}g_{3\sharp}"] \\
                                                                                 & g_1^{\ast}f_2^{\ast}g_4^{\ast}E_6 \arrow[d,equals]                                                               &                                                   \\
f_1^{\ast}g_2^{\ast}g_4^{\ast}E_6 \arrow[r,equals]                                      & g_1^{\ast}g_3^{\ast}f_3^{\ast} \arrow[r, "g_1^{\ast}g_3^{\ast}f_3^{\sharp}"']                             & g_1^{\ast}g_3^{\ast}                             
\end{tikzcd}
    \]
\end{itemize}
\end{defn}

\begin{thm}
For an indexed category $\Ea : \Ba\op \to \Cat$, there is an equivalence of
double categories between the Grothendieck double construction of $\Ea$ and the double category of
spans in its Grothendieck construction with left leg vertical and right leg horizontal.
\end{thm}
\begin{proof}
By Theorem \ref{thm:fiberwise.op.square}, the vertical and horizontal categories
of these two double categories are equivalent. It remains to show that there
exists a map of spans
\begin{equation}\label{eqn:map.of.spans}
    \begin{tikzcd}
      \lens{E_1}{B_1} \arrow[r, shift left, "\lens{g_{1\sharp}}{g_1}"]\arrow[r, shift right] \arrow[d, leftarrow, shift left] \arrow[d, shift right, equals, "\lens{f_1^{\sharp}}{\id}"'] & \lens{E_2}{B_2} \arrow[d, leftarrow, shift left, "\lens{f_2^{\sharp}}{\id}"] \arrow[d, equals, shift right] \\
      \lens{f_1^{\ast}E_3}{B_1} \arrow[d, equals, shift left] \arrow[d, shift right,  "\lens{\id}{f_1}"']\arrow[r, shift left, dashed]\arrow[r, shift right, dashed] & \lens{f_2^{\ast}E_4}{B_2} \arrow[d, equals, shift left, "\lens{\id}{f_2}"] \arrow[d, shift right] \\
        \lens{E_3}{B_3} \arrow[r, shift left]\arrow[r, shift right, "\lens{g_{2\sharp}}{g_2}"'] & \lens{E_4}{B_4}
    \end{tikzcd}
\end{equation}
if and only if the appropriate diagrams as in Diagram
\ref{eqn:groth.double.diagram} commute. We note that, by Lemma
\ref{lem:span.morphism.uniqueness}, such a map of spans is unique if it exists.

A dashed map as in Diagram \ref{eqn:map.of.spans} consists of a component $x :
B_1 \to B_2$ and $y : f_1^{\ast} E_1 \to x^{\ast}f_2^{\ast}E_4$. We consider
what the diagram says about $x$ first. By the commutativity of the top square,
$x$ must equal $g_1$, and therefore the bottom square expresses the
commutativity of the left square of Diagram \ref{eqn:groth.double.diagram}.

Now, $y$ must have signature $f_1^{\ast}E_3 \to g_1^{\ast}f_2^{\ast}E_4$, or
equivalently $f_1^{\ast} E_1 \to f_1^{\ast} g_2^{\ast} E_4$ by the fact that
$f_2 g_1 = g_2 f_1$. The bottom square now says that $y =
f_1^{\ast}g_{2\sharp}$, so that the top square now gives us the right square of
Diagram \ref{eqn:groth.double.diagram}.

Of course, if the squares of Diagram \ref{eqn:groth.double.diagram} commute,
then we can make these choice of $x$ and $y$ in order to give such a morphism of spans.
\end{proof}

There is a useful corollary of this result: any pullback preserving functor $F :
\Ea \to \Ca$ can extend to a double functor $\Ea \wr \Ea^{\vee} \to \type{Span}(\Ca)$.

Finally, we show that this construction is 2-functorial.

\begin{thm}\label{thm:double.functoriality}
The assignment $\Ca \mapsto \Ca \wr \Ca^{\vee}$ sending a cartesian
factorization system to the double category of spans with left leg vertical and
right leg cartesian gives a 2-functor $\type{Cart} \to \type{Dbl}$ from the
2-category of cartesian factorization systems and functors which preserve both
classes to the 2-category of double categories, functors, and horizontal transformations.
\end{thm}
\begin{proof}
A functor $F : \Ca_1 \to \Ca_2$ which preserves both classes with therefore
preserve pullbacks of the form 
\[
\begin{tikzcd}
\bullet \arrow[d, two heads] \arrow[r, tail]  & \bullet \arrow[d, two heads] \\
\bullet \arrow[r, tail]                                 & \bullet                     
\end{tikzcd}
\]
by Lemma \ref{lem:cartesian.fact.pullback.square}. The assignment $\Ca \wr
\Ca^{\vee}$ is therefore transparently functorial.

Given a natural transformation $\alpha : F \Rightarrow G : \Ca_0 \to \Ca_1$, we may construct a
horizontal natural transformation $\overline{\alpha} : F \wr F^{\vee}
\Rightarrow G \wr G^{\vee}$ using the action of $\alpha$. To every object $C$ of
$\Ca_0 \wr \Ca_0^{\vee}$, we assign the horizontal map $\alpha_C : FC \to GC$.
To every span $C \twoheadleftarrow Z \hookrightarrow C'$ we assign the map of
spans:
\[
\begin{tikzcd}
FC \arrow[r, "\alpha_C"]                                      & GC                                      \\
FZ \arrow[u, two heads] \arrow[d, tail] \arrow[r, "\alpha_Z"] & GZ \arrow[u, two heads] \arrow[d, tail] \\
FC' \arrow[r, "\alpha_{C'}"]                                  & GC'                                    
\end{tikzcd}
\]
These satisfy the required laws quite straightforwardly from the naturality of
$\alpha$ and the uniqueness part of the universal property of pullbacks.
\end{proof}

\printbibliography

\end{document}